\documentclass[10pt,reqno]{amsart}
\usepackage{amsmath,amssymb}
\usepackage{enumitem}
\usepackage{graphics,color}
\usepackage{epsfig}
\usepackage{graphicx}
\usepackage{amsfonts,amsmath,amssymb,amscd}
\setbox0=\hbox{$+$}
\newdimen\plusheight
\plusheight=\ht0
\def\+{\;\lower\plusheight\hbox{$+$}\;}

\setbox0=\hbox{$-$}
\newdimen\minusheight
\minusheight=\ht0
\def\-{\;\lower\minusheight\hbox{$-$}\;}

\setbox0=\hbox{$\cdots$}
\newdimen\cdotsheight
\cdotsheight=\plusheight%\ht0
\def\cds{\lower\cdotsheight\hbox{$\cdots$}}

\renewcommand{\(}{\left\(}
\renewcommand{\)}{\right\)}
\renewcommand{\[}{\left[}
\renewcommand{\]}{\right]}

\theoremstyle{plain}
\newtheorem{theorem}{Theorem}[section]
\newtheorem{lemma}[theorem]{Lemma}

\theoremstyle{definition}

\theoremstyle{remark}

\numberwithin{equation}{section}

\begin{document}
\title{ Generalizations of two hypergeometric sums related to conjectures of Guo}
\author{ Arijit Jana and Liton Karmakar}
\address{Department of Mathematics, National Institute of Technology, Silchar, 788010, India}
\email{jana94arijit@gmail.com}

\address{Department of Mathematics, National Institute of Technology, Silchar, 788010, India}
\email{litonofficial8638@gmail.com}
\begin{abstract}
	In $2021$, the first author and Kalita obtained two general hypergeometric formulas for sums involving certain rising factorials  to prove some supercongruence conjectures of Guo related to (B.2) and (C.2).	In this paper, we further  generalize  those formulas  by using the WZ-method and the Zeilberger algorithm respectively. 
	%In \cite{jana8}, first author and Kalita obtained two general hypergeometric formulas to prove some supercongruence conjectures of Guo \cite{guo}. Motivated by their work, we here deduce the further generalizations of those two identities.

	% Motivated by some supercongruence conjectures of Guo \cite{guo}, Jana and Kalita \cite{jana8} obtained two closed hypergeometric  formulas. Here, we give generalizations of those two results.
\end{abstract}

\maketitle
\section{Introduction and statement of results}
In $1914$, Ramanujan \cite{ramanujan} developed $17$ infinite series representations of $\frac{1}{\pi}$, including the following identity
\begin{align} \label{ident1}
\sum_{n=0}^{\infty}\frac{4n+1}{(-64)^n}\binom{2n}{n}^3 = \frac{2}{\pi}.
\end{align} 
 Throughout the paper, let $p$ be an odd prime and $r\geq 1$ be integers. In \cite{hamme}, Van Hamme listed a number of conjectural supercongruences relating those Ramanujan's formulas for $\frac{1}{\pi},$ like \eqref{ident1}, has a nice $p$-adic analogue 
\begin{align}\label{padic1}
(B.2)\hspace{2.2cm} \sum_{n=0}^{\frac{p-1}{2}}\frac{4n+1}{(-64)^n}\binom{2n}{n}^3\equiv(-1)^{\frac{p-1}{2}}p~~(\text{mod }p^3).
\end{align}
The supercongruence \eqref{padic1} was first proved by Mortenson  \cite{mortenson} using a $_6F_5$ hypergeometric  identity, and later reconfirmed by Zudilin \cite{zudilin} using the WZ-method \cite{wilf}, and by Long \cite{long1} through certain hypergeometric series identities and evaluations. In \cite{guo}, Guo  proved  the following  generalization of \eqref{padic1} by providing a $q$-analogue of \eqref{padic1} using $q$-WZ method
\begin{align}\label{padic2}
\sum_{n=0}^{\frac{p^r-1}{2}}\frac{4n+1}{(-64)^n}\binom{2n}{n}^3\equiv p^r(-1)^{\frac{(p-1)r}{2}}~~(\text{mod }p^{r+2}).
\end{align}
Employing the powerful WZ-method, Guo further gave the generalization of \eqref{padic1} as follows
\begin{align}\label{eq2}
\sum_{n=0}^{\frac{p-1}{2}}\frac{(4n+1)^3}{(-64)^n}\binom{2n}{n}^3\equiv-3p(-1)^{\frac{(p-1)}{2}}~~(\text{mod }p^2).
\end{align}
Following this, first author and Kalita \cite{jana8} deduced the further generalization of \eqref{eq2} as
\begin{align}\label{eq3}
\sum_{n=0}^{\frac{p^r-1}{2}}\frac{(4n+1)^3}{(-64)^n}\binom{2n}{n}^3 \equiv -3p^r(-1)^{\frac{(p-1)r}{2}}~~(\text{mod }p^3).
\end{align}
Since $(4n+1)^3+3(4n+1)=4(4n+1)(4n^2+2n+1)$, first author and Kalita \cite{jana1}  remarked from \eqref{padic2} and \eqref{eq3} that 
\begin{align}\label{eq5}
\sum_{n=0}^{\frac{p^r-1}{2}}\frac{(4n+1)(4n^2+2n+1)}{(-64)^n}\binom{2n}{n}^3\equiv 0~~(\text{mod }p^{r+2}).
\end{align}
In this context, Guo \cite[Conjecture 4.2]{guo} posed a stronger supercongruence conjecture for \eqref{eq5} as follows:

\begin{align}
\left\{
\begin{array}{ll}
\displaystyle \sum_{n=0}^{\frac{p^{r}-1}{2}} \frac{(4n+1)(4n^2+2n+1)}{(-64)^{n}} \binom{2n}{n}^{3} \equiv p^{3r}~~(\emph{mod}~~ p^{3r+1}),\\
\displaystyle 	\sum_{n=0}^{p^{r}-1} \frac{(4n+1)(4n^2+2n+1)}{(-64)^{n}} \binom{2n}{n}^{3} \equiv p^{3r}~~(\emph{mod}~~ p^{3r+1}).
		\end{array}
	\right.
\end{align}

%\begin{conjecture}\cite[Conjecture 4.2]{guo}\label{conjguo1}
%	Let $p$ be any odd prime and $r$  a positive integer. Then
%	\begin{align*}
%	&	\sum_{k=0}^{\frac{p^{r}-1}{2}} \frac{(4k+1)(4k^2+2k+1)}{(-64)^{k}} \binom{2k}{k}^{3} \equiv p^{3r}~~(\emph{mod}~~ p^{3r+1}),\\
%	&	\sum_{k=0}^{p^{r}-1} \frac{(4k+1)(4k^2+2k+1)}{(-64)^{k}} \binom{2k}{k}^{3} \equiv p^{3r}~~(\emph{mod}~~ p^{3r+1}).
%	\end{align*}
%\end{conjecture}
To prove the  above Conjecture, first author and Kalita \cite{jana8} gave a more general result using telescopic method as follows
\begin{align} \label{res1}
\sum_{n=0}^{M}(-1)^{n}(2\ell n+1)(\ell^2 n^2 +\ell n +1) \frac{(\frac{1}{\ell})_{n}^{3}}{(1)_n^{3}}= (-1)^{M} \frac{\left(1+\frac{1}{\ell }\right)_M^3}{(1)_M^3},
\end{align}
	where	the  rising factorial  $(a)_n$ is defined by
$$(a)_0=1,~~~(a)_n = a(a+1)\cdots (a+n-1) \text{ for } n\geq 1.$$
The rising factorial for negative index $\[\ref{gasper}, (1.2.27)\]$ is defined  as 
$$ \left(a\right)_{-n}= \frac{1}{\left(a-1\right) \left(a-2\right) \cdots \left(a-n\right)} = \frac{1}{\left(a-n\right)_{n}}= \frac{\left(-1\right)^n}{\left(1-a\right)_{n}}. $$

Motivated by the above result,  we here  give a further generalization of \eqref{res1} by giving a new WZ-pair in the following theorem. 
\begin{theorem}\label{thmguo1a}
Let $\ell \geq 1$, $s \geq 0, M \geq 0$ with  $M \geq s$ be integers. Then
\begin{align*}
&\displaystyle \sum_{n=s}^{M} \left(-1\right)^{n} \left(2\ell n +1\right) \left(\ell^2n^2 + \ell n + 1 -\ell^2 s^2 \right) \frac{\left(\frac{1}{\ell}\right)_{n+s} \left(\frac{1}{\ell}\right)_{n-s} \left(\frac{1}{\ell}\right)_{n}}{\left(1\right)_{n+s} \left(1\right)_{n-s} \left(1\right)_{n}}\\
&= \left(-1\right)^{M} \frac{\left(1+\frac{1}{\ell}\right)_{M+s} \left(1+\frac{1}{\ell}\right)_{M-s} \left(1+\frac{1}{\ell}\right)_{M}}{\left(1\right)_{M+s} \left(1\right)_{M-s} \left(1\right)_{M}}.
\end{align*}

\end{theorem}

\par  In $2011,$ using hypergeometric identities, Long \cite{long1} proved that for prime $p\geq 5$
\begin{align}\label{padic3}
(C.2)\hspace{2.2cm}\sum_{n=0}^{\frac{p-1}{2}}\frac{4n+1}{256^n}\binom{2n}{n}^4\equiv p ~~(\text{mod } p^{4}).
\end{align}
Then, Guo and Wang \cite{guo1} gave the following generalization  of \eqref{padic3} by establishing a $q$-analogue
\begin{align}\label{eq6}
\sum_{n=0}^{\frac{p^r-1}{2}}\frac{4n+1}{256^n}\binom{2n}{n}^4\equiv p^r ~~(\text{mod } p^{r+3}).
\end{align}
Using hypergeometric series identities and evaluations, first author and Kalita \cite{jana3} have also given a proof of \eqref{eq6}. In the same paper, first author and Kalita proved the following generalization of \eqref{eq6} as
\begin{align}\label{eq7}
\sum_{n=0}^{\frac{p^r-1}{2}}\frac{(4n+1)^3}{256^n}\binom{2n}{n}^4\equiv -p^r ~~(\text{mod } p^{r+3}).
\end{align}
In this context, Guo \cite[Conjecture 4.4]{guo} has provided another supercongruence  conjecture  related to \eqref{eq6} and \eqref{eq7} as follows:
\begin{align}
\left\{
\begin{array}{ll}
\displaystyle 	\sum_{n=0}^{\frac{p^{r}-1}{2}} \frac{(4n+1)(8n^2+4n+1)}{256^{n}} \binom{2n}{n}^{4} \equiv p^{4r}~~(\emph{mod}~~ p^{4r+1}),\\
\displaystyle 	\sum_{n=0}^{p^{r}-1} \frac{(4n+1)(8n^2+4n+1)}{256^{n}} \binom{2n}{n}^{4} \equiv p^{4r}~~(\emph{mod}~~ p^{4r+1}).
\end{array}
\right.
\end{align}

To prove the above Conjecture, first author and Kalita \cite{jana8}  deduced a more general result in the closed form via telescopic method as follows 
\begin{align} \label{res2}
\sum_{n=0}^{M}(2\ell n+1)(2\ell^2 n^2 +2\ell n +1) \frac{(\frac{1}{\ell})_{n}^{4}}{(1)_n^{4}}= \frac{\left(1+\frac{1}{\ell }\right)_M^4}{(1)_M^4}.
\end{align}
Motivated by the above result, we here give a further generalization of \eqref{res2} as 
\begin{theorem}\label{thmguo1b}
	Let $\ell \geq 1$, $s, M \geq 0$ with  $M \geq s$ be integers. Then
	\begin{center}
		$\displaystyle \sum_{n=s}^{M} \left(2\ell n +1\right) \left(2\ell^2n^2 + 2\ell n +1 -\ell^2 s^2 \right) \frac{\left(\frac{1}{\ell}\right)_{n+s} \left(\frac{1}{\ell}\right)_{n-s} \left(\frac{1}{\ell}\right)_{n}^2}{\left(1\right)_{n+s} \left(1\right)_{n-s} \left(1\right)_{n}^2}=  \frac{\left(1+\frac{1}{\ell}\right)_{M+s} \left(1+\frac{1}{\ell}\right)_{M-s} \left(1+\frac{1}{\ell}\right)_{M}^2}{\left(1\right)_{M+s} \left(1\right)_{M-s} \left(1\right)_{M}^2}.$
	\end{center}
\end{theorem}

 \par The rising factorial $\frac{(\frac{1}{\ell})_{n}}{(1)_n}$ appears in a number of hypergeometric closed-sum formulas.
 In the literature, we have not found any general identities that resemble the sum containing $ \frac{\left(\frac{1}{\ell}\right)_{n+s} \left(\frac{1}{\ell}\right)_{n-s} }{\left(1\right)_{n+s} \left(1\right)_{n-s} }.$
 The paper is organised as follows. In Section $2$, we shall first give  a new WZ-pair via the WZ-method designed by Wilf and Zeilberger \cite{wilf}, and then using this we prove Theorem \ref{thmguo1a}. In Section $3$,  we prove  Theorem \ref{thmguo1b} by using Zeilberger algorithm.
\section{WZ-method and Proof of Theorem \ref{thmguo1a} }

Wilf-Zeilberger method is based on preliminary human guess to finding a WZ-pair. Motivated by the work of \cite{ekhad,guillera1, guillera2,zudilin}, we here employ the below hypergeometric function to deduce the following lemma
\begin{center}
	$\displaystyle F(n,k)=  (-1)^{n+k}(2\ell n+1) \frac{{\left (\frac{1}{\ell} \right)}_{n+s}{\left (\frac{1}{\ell} \right)}_{n-s}{\left (\frac{1}{\ell} \right)}_{n+k}}{{\left (1 \right)}_{n+s}{\left (1 \right)}_{n-s}{\left (1 \right)}_{n-k}{\left (\frac{1}{\ell} \right)}_{k+s}{\left (\frac{1}{\ell} \right)}_{k-s}}.$
\end{center}

\begin{lemma} \label{Lemma1}
	Let $\ell \geq 1$, $n, s, k \geq 0$ with $n \geq s$ and $n \geq k$ be integers. Suppose
	\begin{center}
		$\displaystyle F(n,k)=  (-1)^{n+k}(2\ell n+1) \frac{{\left (\frac{1}{\ell} \right)}_{n+s}{\left (\frac{1}{\ell} \right)}_{n-s}{\left (\frac{1}{\ell} \right)}_{n+k}}{{\left (1 \right)}_{n+s}{\left (1 \right)}_{n-s}{\left (1 \right)}_{n-k}{\left (\frac{1}{\ell} \right)}_{k+s}{\left (\frac{1}{\ell} \right)}_{k-s}}$
	\end{center}
	and
	\begin{center}
		$\displaystyle G(n,k)=  (-1)^{n+k} \ell \frac{{\left (\frac{1}{\ell} \right)}_{n+s}{\left (\frac{1}{\ell} \right)}_{n-s}{\left (\frac{1}{\ell} \right)}_{n+k-1}}{{\left (1 \right)}_{n+s-1}{\left (1 \right)}_{n-s-1}{\left (1 \right)}_{n-k}{\left (\frac{1}{\ell} \right)}_{k+s}{\left (\frac{1}{\ell} \right)}_{k-s}},$	
	\end{center}
	where $1 \slash \left(1\right)_{m}=0$ for $m= -1, -2, \ldots.$ Then
		\begin{equation} \label{Eqn 3.1}
	F(n,k-1) - F(n,k)= G(n+1,k) - G(n,k).
	\end{equation} 
\end{lemma}

\begin{proof}
	We have
	
	\begin{align*}
	\displaystyle \frac{F(n,k-1)}{F(n,k)} & = - \frac{\left(k+s+ \frac{1-\ell}{\ell} \right) \left(k-s + \frac{1-\ell}{\ell}\right)}{\left(n+k+ \frac{1-\ell}{\ell}\right) \left(n-k+1 \right)},
	\end{align*}
	\begin{align*}
	\displaystyle \frac{G(n+1,k)}{F(n,k)} & = - \frac{\ell \cdot \left(n+\frac{1}{\ell} + s\right) \left(n + \frac{1}{\ell} -s\right)}{\left(2\ell n + 1\right)\left(n-k+1\right)},
	\end{align*}	
	and
	\begin{align*}
	\displaystyle \frac{G(n,k)}{F(n,k)} & = \frac{\ell \cdot \left(n^{2} - s^{2}\right)}{\left(2 \ell n + 1\right)\left(n+k+ \frac{1-\ell}{\ell}\right)}.
	\end{align*}
	
	Noting that 
	\begin{align*}
	\frac{F(n,k-1)}{F(n,k)} -1 & = - \left\{ \frac{k^{2} + 2k\left(\frac{1-\ell}{\ell}\right) + \left(\frac{1-\ell}{\ell}\right)^2 - s^2  + \left(n+k+ \frac{1-\ell}{\ell}\right) \left(n-k+1 \right) }{\left(n+k+ \frac{1-\ell}{\ell}\right) \left(n-k+1 \right)}\right\},
	\end{align*}
	and
	\begin{align*}
	\frac{G(n+1,k)}{F(n,k)} - \frac{G(n,k)}{F(n,k)} & = - \ell \cdot \left\{ \frac {\left(n^2+\frac{2n}{\ell}+\frac{1}{\ell^{2}}-s^2\right) \left(n+k+ \frac{1-\ell}{\ell}\right) + \left(n^2-s^2\right) \left(n-k+1\right)  }{\left(2\ell n+1\right)\left(n+k+ \frac{1-\ell}{\ell}\right) \left(n-k+1 \right)}\right\}.
	\end{align*}
	It is easy to obtain that
	\begin{center}
		$\frac{F(n,k-1)}{F(n,k)}-1= \frac{G(n+1,k)}{F(n,k)} - \frac{G(n,k)}{F(n,k)},$
	\end{center}
	and hence
	\begin{center}
		$F(n,k-1) - F(n,k)= G(n+1,k) - G(n,k).$
	\end{center}
	\end{proof}

\begin{proof}[Proof of theorem~\ref{thmguo1a}]

	Multiplying both sides of  \eqref{Eqn 3.1} by $\ell^2 \left(\frac{1}{\ell}\right)_{k+s} \left(\frac{1}{\ell}\right)_{k-s},$ we get
	\begin{align*}
	\displaystyle \ell^2 \left(\frac{1}{\ell}\right)_{k+s} \left(\frac{1}{\ell}\right)_{k-s} \left\{F(n,k-1) - F(n,k)\right\} & = \ell^2 \left(\frac{1}{\ell}\right)_{k+s} \left(\frac{1}{\ell}\right)_{k-s} \left\{G(n+1,k) - G(n,k)\right\}.
	\end{align*}
	Substituting $k=1$ and then taking sum on both sides with $n$ ranging from $s$ to $M$, we obtain
	\begin{align*}
	\displaystyle  \sum_{n=s}^{M} \ell^2 \left(\frac{1}{\ell}\right)_{1+s} \left(\frac{1}{\ell}\right)_{1-s} \left\{F(n,0) - F(n,1)\right\}  = \ell^2 \left(\frac{1}{\ell}\right)_{1+s} \left(\frac{1}{\ell}\right)_{1-s} \left\{G(M+1,1) - G(s,1)\right\}. 
	\end{align*}
		We have 
		\begin{align*}
			&\displaystyle  \sum_{n=s}^{M} \ell^2 \left(\frac{1}{\ell}\right)_{1+s} \left(\frac{1}{\ell}\right)_{1-s} \left\{F(n,0) - F(n,1)\right\} \\
			 & = \displaystyle  \sum_{n=s}^{M} \ell^2 \left(\frac{1}{\ell}\right)_{1+s}      \left(\frac{1}{\ell}\right)_{1-s} \left\{ \frac{{(-1)^{n}(2\ell n+1)\left (\frac{1}{\ell} \right)}_{n+s}{\left (\frac{1}{\ell} \right)}_{n-s}{\left (\frac{1}{\ell} \right)}_{n}}{{\left (1 \right)}_{n+s}{\left (1 \right)}_{n-s}{\left (1 \right)}_{n}{\left (\frac{1}{\ell} \right)}_{s}{\left (\frac{1}{\ell} \right)}_{-s}}  \right. + \left.    \frac{{(-1)^{n}(2\ell n+1) \left (\frac{1}{\ell} \right)}_{n+s}{\left (\frac{1}{\ell} \right)}_{n-s}{\left (\frac{1}{\ell} \right)}_{n+1}}{{\left (1 \right)}_{n+s}{\left (1 \right)}_{n-s}{\left (1 \right)}_{n-1}{\left (\frac{1}{\ell} \right)}_{1+s}{\left (\frac{1}{\ell} \right)}_{1-s}}      \right\}\\
			& =\sum_{n=s}^{M} \frac{\left(-1\right)^n \left(2 \ell n +1\right) \ell^2 \left(\frac{1}{\ell}\right)_{n+s} \left(\frac{1}{\ell}\right)_{n-s} \left(\frac{1}{\ell}\right)_{n}}{\left(1\right)_{n+s} \left(1\right)_{n-s} \left(1\right)_{n}} \left\{\frac{\left(\frac{1}{\ell}\right)_{1+s} \left(\frac{1}{\ell}\right)_{1-s}}{\left(\frac{1}{\ell}\right)_{s} \left(\frac{1}{\ell}\right)_{-s}} \right. + \left. \frac{\left(\frac{1}{\ell}\right)_{n+1} \left(1\right)_{n}}{\left(1\right)_{n-1} \left(\frac{1}{\ell}\right)_{n}}  \right\} \\
			& = \sum_{n=s}^{M} \frac{\left(-1\right)^n \left(2 \ell n +1\right) \ell^2 \left(\frac{1}{\ell}\right)_{n+s} \left(\frac{1}{\ell}\right)_{n-s} \left(\frac{1}{\ell}\right)_{n}}{\left(1\right)_{n+s} \left(1\right)_{n-s} \left(1\right)_{n}} \left\{\left(\frac{1}{\ell} + s\right) \left(\frac{1}{\ell} -s\right) \right. + \left. n \left(\frac{1}{\ell} + n\right)\right\} \\
			& = \displaystyle \sum_{n=s}^{M} \left(-1\right)^{n} \left(2\ell n +1\right) \left(\ell^2n^2 + \ell n + 1 -\ell^2 s^2 \right) \frac{\left(\frac{1}{\ell}\right)_{n+s} \left(\frac{1}{\ell}\right)_{n-s} \left(\frac{1}{\ell}\right)_{n}}{\left(1\right)_{n+s} \left(1\right)_{n-s} \left(1\right)_{n}}.
		\end{align*}

	Since $\frac{1}{\left(1\right)_{-1}}=0,$ we have $G(s,1)=0.$ As a result,
	\begin{align*}
	\ell^2 \left(\frac{1}{\ell}\right)_{1+s} \left(\frac{1}{\ell}\right)_{1-s} \left\{G(M+1,1) - G(s,1)\right\} &  = \left(-1\right)^{M} \ell^3 \frac{\left(\frac{1}{\ell}\right)_{M+s+1} \left(\frac{1}{\ell}\right)_{M-s+1} \left(\frac{1}{\ell}\right)_{M+1}}{\left(1\right)_{M+s} \left(1\right)_{M-s} \left(1\right)_{M}}   \\
	& = \left(-1\right)^{M} \frac{\left(1+\frac{1}{\ell}\right)_{M+s} \left(1+\frac{1}{\ell}\right)_{M-s} \left(1+\frac{1}{\ell}\right)_{M}}{\left(1\right)_{M+s} \left(1\right)_{M-s} \left(1\right)_{M}}.
	\end{align*}
	This completes the proof.
\end{proof}
\section{ Zeilberger algorithm and Proof of Theorem  \ref{thmguo1b}}

For integers $n\geq k \geq 0$, suppose $F(n,k)$ is a hypergeometric function in $n$ and $k$. Employing Zeilberger's algorithm, one can find another hypergeometric function $G(n,k)$, as well as polynomials $p(k)$ and $ q(k)$ such that
$$ p(k)	F(n,k-1)-q(k)F(n,k)= G(n+1, k)-G(n,k).$$ However the process is not so obvious always.
% If $p(k) = q(k-1)$, then $q(k)F(n, k)$ and $G(n, k)$ form a WZ-pair.
  We here use the Zeilberger's algorithm for the function
\begin{center}
	$\displaystyle F(n,k)=  (-1)^{k}(2\ell n +1) \frac{{\left (\frac{1}{\ell} \right)}_{n+s}{\left (\frac{1}{\ell} \right)}_{n-s}{\left(\frac{1}{\ell}\right)_{n}}{\left (\frac{1}{\ell} \right)}_{n+k}}{{\left (1 \right)}_{n+s}{\left (1 \right)}_{n-s}{\left(1\right)_{n}}{\left (1 \right)}_{n-k}{\left (\frac{1}{\ell} \right)}_{k+s}{\left (\frac{1}{\ell} \right)}_{k-s}}.$
\end{center}
 Our motivation of taking the above hypergeometric function is based on WZ-pairs of \cite{ekhad,guillera1,guillera2,wang}.

\begin{lemma} \label{Lemma2}
	Let $\ell \geq 1$, $n, s, k \geq 0$ with $n \geq s$ and $n \geq k$ be integers. Suppose
	\begin{center}
		$\displaystyle F(n,k)=  (-1)^{k}(2\ell n +1) \frac{{\left (\frac{1}{\ell} \right)}_{n+s}{\left (\frac{1}{\ell} \right)}_{n-s}{\left(\frac{1}{\ell}\right)_{n}}{\left (\frac{1}{\ell} \right)}_{n+k}}{{\left (1 \right)}_{n+s}{\left (1 \right)}_{n-s}{\left(1\right)_{n}}{\left (1 \right)}_{n-k}{\left (\frac{1}{\ell} \right)}_{k+s}{\left (\frac{1}{\ell} \right)}_{k-s}}$
	\end{center}
	and
	\begin{center}
		$\displaystyle G(n,k)=  (-1)^{k-1} {\ell}^2 \frac{{\left (\frac{1}{\ell} \right)}_{n+s}{\left (\frac{1}{\ell} \right)}_{n-s}{\left(\frac{1}{\ell}\right)_{n}}{\left (\frac{1}{\ell} \right)}_{n+k-1}}{{\left (1 \right)}_{n+s-1}{\left (1 \right)}_{n-s-1}{\left(1\right)_{n-1}}{\left (1 \right)}_{n-k}{\left (\frac{1}{\ell} \right)}_{k+s}{\left (\frac{1}{\ell} \right)}_{k-s}},$	
	\end{center}
	where $1 \slash \left(1\right)_{m}=0$ for $m= -1, -2, \ldots.$ Then
		\begin{equation} \label{Eqn 3.2}
	\left(\ell k - \ell + 1\right)F(n,k-1) - \left(\ell k -\ell + 2\right)F(n,k)= G(n+1,k) - G(n,k).
	\end{equation} 
	
\end{lemma}
\begin{proof}
	We have
	\begin{align*}
	\displaystyle \frac{F(n,k-1)}{F(n,k)} & = - \frac{\left(k+s+ \frac{1-\ell}{\ell} \right) \left(k-s + \frac{1-\ell}{\ell}\right)}{\left(n+k+ \frac{1-\ell}{\ell}\right) \left(n-k+1 \right)},
	\end{align*}
	
	\begin{align*}
	\displaystyle \frac{G(n+1,k)}{F(n,k)} & = - \frac{\ell^{2} \cdot \left(n+\frac{1}{\ell} + s\right) \left(n + \frac{1}{\ell} -s\right) \left(\frac{1}{\ell}+n\right)}{\left(2\ell n + 1\right)\left(n-k+1\right)},
	\end{align*}
	and
	\begin{align*}
	\displaystyle \frac{G(n,k)}{F(n,k)} & = -\frac{\ell^2 \cdot \left(n^{3} - ns^{2}\right)}{\left(2\ell n +1\right)\left(n+k+ \frac{1-\ell}{\ell}\right)}.
	\end{align*}
	Noting that
	
	\begin{center}
		$\left(\ell k -\ell +1\right)\frac{F(n,k-1)}{F(n,k)} -\left(\ell k -\ell +2\right)  = - \left\{ \frac{\left(k+s+\frac{1-\ell}{\ell}\right)\left(k-s+\frac{1-\ell}{\ell}\right)\left(\ell k -\ell +1\right)  + \left(n+k+ \frac{1-\ell}{\ell}\right) \left(n-k+1 \right)\left(\ell k -\ell + 2 \right) }{\left(n+k+ \frac{1-\ell}{\ell}\right) \left(n-k+1 \right)}\right\},$
	\end{center}
	and
	\begin{align*}
	\frac{G(n+1,k)}{F(n,k)} - \frac{G(n,k)}{F(n,k)} & = - \ell^2 \cdot \left\{\frac{ \left(n^2+\frac{2n}{\ell}+\frac{1}{\ell^{2}}-s^2\right) \left(n+k+ \frac{1-\ell}{\ell}\right)\left(\frac{1}{\ell}+n\right) - \left(n^3-ns^2\right) \left(n-k+1\right) }{\left(2\ell n+1\right)\left(n+k+ \frac{1-\ell}{\ell}\right) \left(n-k+1 \right)}\right\}.
	\end{align*}
	It is easy to obtain that
	\begin{center}
		$\left(\ell k -\ell +1\right)\frac{F(n,k-1)}{F(n,k)} -\left(\ell k -\ell +2\right) = \frac{G(n+1,k)}{F(n,k)} - \frac{G(n,k)}{F(n,k)},$
	\end{center}
	and hence
	\begin{center}
		$\left(\ell k - \ell + 1\right)F(n,k-1) - \left(\ell k -\ell + 2\right)F(n,k)= G(n+1,k) - G(n,k).$
	\end{center}
	
\end{proof}

\begin{proof}[Proof of theorem~\ref{thmguo1b}]

	Multiplying both sides of  \eqref{Eqn 3.2} by $\ell^2 \left(\frac{1}{\ell}\right)_{k+s} \left(\frac{1}{\ell}\right)_{k-s},$ we get
	\begin{align*}
	&\displaystyle \ell^2 \left(\frac{1}{\ell}\right)_{k+s} \left(\frac{1}{\ell}\right)_{k-s} \left\{\left(\ell k - \ell + 1\right)F(n,k-1) - \left(\ell k -\ell + 2\right)F(n,k)\right\} \\
	& = \ell^2 \left(\frac{1}{\ell}\right)_{k+s} \left(\frac{1}{\ell}\right)_{k-s} \left\{G(n+1,k) - G(n,k)\right\}.
	\end{align*}
	Substituting $k=1$ and then taking sum on both sides with $n$ ranging from $s$ to $M$, we obtain
	\begin{align*}
	\displaystyle  \sum_{n=s}^{M} \ell^2 \left(\frac{1}{\ell}\right)_{1+s} \left(\frac{1}{\ell}\right)_{1-s} \left\{F(n,0) - 2F(n,1)\right\}  = \ell^2 \left(\frac{1}{\ell}\right)_{1+s} \left(\frac{1}{\ell}\right)_{1-s} \left\{G(M+1,1) - G(s,1)\right\}. 
	\end{align*}
	One can easily find that 
	\begin{align*}
	\displaystyle  \sum_{n=s}^{M} \ell^2 \left(\frac{1}{\ell}\right)_{1+s} \left(\frac{1}{\ell}\right)_{1-s} \left\{F(n,0) - 2F(n,1)\right\}  = \displaystyle \sum_{n=s}^{M} \left(2\ell n +1\right) \left(2\ell^2n^2 + 2\ell n +1 -\ell^2 s^2 \right) \frac{\left(\frac{1}{\ell}\right)_{n+s} \left(\frac{1}{\ell}\right)_{n-s} \left(\frac{1}{\ell}\right)_{n}^2}{\left(1\right)_{n+s} \left(1\right)_{n-s} \left(1\right)_{n}^2}.
	\end{align*}
	Again,
	\begin{align*}
	\ell^2 \left(\frac{1}{\ell}\right)_{1+s} \left(\frac{1}{\ell}\right)_{1-s} \left\{G(M+1,1) - G(s,1)\right\} &  = \ell^4 \frac{\left(\frac{1}{\ell}\right)_{M+s+1} \left(\frac{1}{\ell}\right)_{M-s+1} \left(\frac{1}{\ell}\right)_{M+1}^2}{\left(1\right)_{M+s} \left(1\right)_{M-s} \left(1\right)_{M}^2}   \\
	& =  \frac{\left(1+\frac{1}{\ell}\right)_{M+s} \left(1+\frac{1}{\ell}\right)_{M-s} \left(1+\frac{1}{\ell}\right)_{M}^2}{\left(1\right)_{M+s} \left(1\right)_{M-s} \left(1\right)_{M}^2}.
	\end{align*}
	This completes the proof.
\end{proof}

\section{Acknowledgement} 
The second author has been supported by the institute doctoral fellowship of National Institute of Technology, Silchar.

\end{document}